\documentclass[12pt,reqno]{amsart}

\usepackage{amssymb}

\newtheorem{theorem}{Theorem}[section]
\newtheorem{proposition}[theorem]{Proposition}
\newtheorem{lemma}[theorem]{Lemma}
\newtheorem{corollary}[theorem]{Corollary}

\theoremstyle{definition}
\newtheorem{definition}[theorem]{Definition}
\newtheorem{remark}[theorem]{Remark}

\numberwithin{equation}{section}

\begin{document}

\baselineskip=15pt

\title[Parabolic ${\rm SL}_r$--opers]{Parabolic ${\rm SL}_r$--opers}

\author[I. Biswas]{Indranil Biswas}

\address{School of Mathematics, Tata Institute of Fundamental
Research, Homi Bhabha Road, Mumbai 400005, India}

\email{indranil@math.tifr.res.in}

\author[S. Dumitrescu]{Sorin Dumitrescu}

\address{Universit\'e C\^ote d'Azur, CNRS, LJAD, France}

\email{dumitres@unice.fr}

\author[C. Pauly]{Christian Pauly}

\address{Universit\'e C\^ote d'Azur, CNRS, LJAD, France}

\email{pauly@unice.fr}



\date{}
\begin{abstract}
We define ${\rm SL}_r$--opers in the set-up of vector bundles on curves with a parabolic
structure over a divisor. Basic properties of these objects are investigated.
\end{abstract}

\maketitle

\section{Introduction}

The notion of oper was introduced by Beilinson and Drinfeld \cite{BD} as an essential ingredient in 
the geometric Langlands program; they were influenced by earlier work of Drinfeld and Sokolov 
\cite{DS}, \cite{DS2}. Since then opers have appeared in numerous contexts and by now this notion has been 
established as an important topic.

Let $X$ be a compact connected Riemann surface and $G$ a semisimple
affine algebraic group defined over $\mathbb C$; fix a Borel subgroup $B\, \subset\, G$. A
$G$--oper on $X$ is a holomorphic principal $G$--bundle $E_G$ over $X$, together with
\begin{itemize}
\item a holomorphic reduction of structure group $E_B\, \subset\, E_G$ of $E_G$ to $B\, \subset\, G$,
and

\item a holomorphic connection on $E_G$ with respect to which the reduction $E_B$ satisfies a
certain transversality condition. The transversality condition in question is
described below for the case of $G\,=\,\text{SL}_r({\mathbb C})$.
\end{itemize}
In case of $G\,=\, \text{SL}_r({\mathbb C})$, an $\text{SL}_r({\mathbb C})$--bundle corresponds to
a holomorphic vector bundle $E$ on $X$ of rank $r$ such that $\bigwedge^r E\,=\, {\mathcal O}_X$.
A $B$--reduction of it is given by a complete flag
$$
0\,=\, E_0\, \subset\, E_1\, \subset\, E_2\, \subset\, \cdots\, \subset\, E_{r-1}\, \subset\, E_r \,=\, E
$$
of holomorphic subbundles such that $\text{rank}(E_i)\,=\, i$. A holomorphic connection $D$ on
this filtered bundle defines an $\text{SL}_r({\mathbb C})$--oper if
\begin{enumerate}
\item $D(E_i)\, \subset\, E_{i+1}\otimes K_X$ for all $i$, where $K_X$ is the holomorphic cotangent
bundle of $X$, 

\item the corresponding second fundamental form
$E_i/E_{i-1}\, \longrightarrow\, (E_{i+1}/E_i)\otimes K_X$ is an isomorphism for all
$1\, \leq\, i\, \leq\, r-1$, and

\item $\bigwedge^r D = d$, the standard de Rham differentiation on ${\mathcal O}_X$.
\end{enumerate}
The underlying holomorphic vector bundle of a $\text{SL}_r({\mathbb C})$--oper is unique up to 
tensor product with a finite set of line bundles of order $r$: the underlying bundle of a 
$\text{SL}_2({\mathbb C})$--oper is the unique nontrivial extension $V_0$ of $K^{-1/2}_X$ by 
$K^{1/2}_X$,  for some theta characteristic $K^{1/2}_X$ on $X$. More generally, the underlying bundle 
of a $\text{SL}_r({\mathbb C})$--oper is, up to tensor product with an $r$-torsion line bundle,  the 
symmetric power $\text{Sym}^{r-1}(V_0)$, where $V_0$ is the underlying bundle of
a $\text{SL}_2({\mathbb C})$--oper.

The isomorphism classes of $\text{SL}_2({\mathbb C})$--opers are in bijection with the
projective structures on $X$; see \cite{Gu} for projective structures on $X$ (see also \cite{GKM}).

Fix finitely many distinct points $x_i$ on $X$, and consider the effective divisor $S\,=\, x_1 + 
\cdots + x_m$. A quasiparabolic structure, over $S$, on a holomorphic vector bundle $E$ on $X$ is a 
decreasing filtration of subspaces of each fiber $E_{x_i}$ (this filtration need not be complete). 
A parabolic structure on $E$ is a quasiparabolic structure together with weights for the subspaces 
that are nonnegative real numbers strictly less than $1$. For any (decreasing) filtration, these 
numbers are actually strictly increasing. A vector bundle with a parabolic structure is called a parabolic 
vector bundle \cite{MS}, \cite{MY}.

Our aim here is to introduce and study the ${\rm SL}_r$--opers in the context of parabolic vector 
bundles.

The article is organized in the following way. Section \ref{se2} presents general facts about jet 
bundles and differential operators on Riemann surfaces. The ${\rm SL}_2$--opers in the parabolic 
set-up are defined and studied in Section \ref{se4} . The main result of that section (Theorem 
\ref{thm1}) says that the space of parabolic ${\rm SL}_2$--opers with fixed parabolic 
weights is an affine space for the vector space $H^0(X, K^2_X \otimes {\mathcal O}_X(S))$. The 
starting point is to actually identify the rank two parabolic bundle underlying parabolic ${\rm 
SL}_2$--opers. The underlying parabolic bundle for ${\rm SL}_2$--opers is investigated in Section 
\ref{se3}. Parabolic ${\rm SL}_r$--opers are defined in Section \ref{se5}. Section \ref{se6} 
focuses on the case where the parabolic weights are rational of special type and
$r$ is odd. In this case we show 
that parabolic ${\rm SL}_r$--opers on $X$ are in natural bijection with  invariant   ${\rm SL}_r$--opers on a 
ramified Galois covering $Y$ over $X$ equipped with an action of the Galois group (see Theorem 
\ref{thm2}). This result uses in an essential way the correspondence studied in \cite{Bi},
\cite{Bo1}, \cite{Bo2}, and only 
works for certain rational parabolic weights under the assumption that
$r$ is odd. In this case we deduce a natural  parametrization of the space of parabolic ${\rm SL}_r$--opers on $X$, with given (special rational)  parabolic weights (see Theorem \ref{propl}).   In the last section we discuss alternative 
definitions of ${\rm SL}_r$--opers and related questions.

Finally we would like to add that similar constructions have recently been carried out by Y. 
Wakabayashi (\cite{W}, Theorem A).

\section{Differential operators}\label{se2}

Let $X$ be a compact connected Riemann surface.
The holomorphic cotangent bundle of $X$ will be denoted by $K_X$.
Let $p_i\, :\, X\times X \, \longrightarrow\, X$ be the projection to
the $i$-th factor, where $i\,=\,1,\, 2$. Let
$$
\Delta\,=\, \{(x,\, x)\, \mid\, x\,\in\, X\}\, \subset\, X\times X
$$
be the reduced diagonal divisor. We shall identify $\Delta$ with $X$
using the map $x\, \longmapsto\, (x,\, x)$.

Take a holomorphic vector bundle $V$ on $X$.
For any integer $k\, \geq\, 0$, the $k$-th jet bundle
$J^k(V)$ for $V$ is defined to be the direct image
$$
J^k(V)\, :=\, p_{1*}((p^*_2 V)/((p^*_2 V)\otimes {\mathcal O}_{X\times X}(-(k+1)\Delta)))
\, \longrightarrow\, X\, .
$$
The line bundle ${\mathcal O}_{X\times X}(-\Delta)\vert_\Delta$ is identified with $K_X$ by 
the Poincar\'e adjunction formula; more precisely, for any holomorphic coordinate
function $z$ defined on any open subset $U\, \subset\, X$, the section of
${\mathcal O}_{X\times X}(-\Delta)$ over $\Delta\bigcap (U\times U)$ given by
$z\circ p_2 - z\circ p_1$ coincides with the section $dz$ of $K_X\vert_U$.
This identification between ${\mathcal O}_{X\times X}(-\Delta)\vert_\Delta$
and $K_X$ produces a short exact sequence of sheaves
\begin{equation}\label{z1}
0\, \longrightarrow\, {\mathcal O}_{X\times X}(-(k+2)\Delta)
\, \longrightarrow\, {\mathcal O}_{X\times X}(-(k+1)\Delta)
\, \longrightarrow\, K^{\otimes (k+1)}_X \, \longrightarrow\, 0
\end{equation}
on $X\times X$, where $K^{\otimes (k+1)}_X$ is supported on
$\Delta$. Consider the short exact sequence of sheaves on $X\times X$
$$
0\, \longrightarrow\, ((p^*_2 V)\otimes {\mathcal O}_{X\times X}(-(k+1)\Delta))/
((p^*_2 V)\otimes {\mathcal O}_{X\times X}(-(k+2)\Delta))\, \longrightarrow\,
$$
\begin{equation}\label{di}
(p^*_2 V)/((p^*_2 V)\otimes {\mathcal O}_{X\times X}(-(k+2)\Delta))
\, \longrightarrow\, (p^*_2 V)/((p^*_2 V)\otimes {\mathcal O}_{X\times X}(-(k+1)\Delta))
\, \longrightarrow\, 0\, .
\end{equation}
Let
$$
0\, \longrightarrow\, p_{1*}(((p^*_2 V)\otimes {\mathcal O}_{X\times X}(-(k+1)\Delta))/
((p^*_2 V)\otimes {\mathcal O}_{X\times X}(-(k+2)\Delta)))\, \longrightarrow\,
$$
$$
p_{1*}((p^*_2 V)/((p^*_2 V)\otimes {\mathcal O}_{X\times X}(-(k+2)\Delta)))
\, \longrightarrow\, p_{1*}((p^*_2 V)/((p^*_2 V)\otimes {\mathcal O}_{X\times X}(-(k+1)\Delta)))
\, \longrightarrow\, 0\, .
$$
the direct image of it on $X$ by the map $p_1$; note that the higher direct images vanish because
the supports of the sheaves in the short exact sequence in \eqref{di} are finite over $X$.
Using \eqref{z1}, this exact sequence of direct images becomes the short exact sequence
of vector bundles
\begin{equation}\label{f1}
0\, \longrightarrow\, V\otimes K^{k+1}_X \, \longrightarrow\,
J^{k+1}(V) \, \stackrel{q^k_V}{\longrightarrow}\, J^k(V) \, \longrightarrow\, 0
\end{equation}
on $X$.

For any holomorphic vector bundle $W$ on $X$, any ${\mathcal O}_X$--linear homomorphism
$\delta\, :\, V\, \longrightarrow\, W$ produces a homomorphism
\begin{equation}\label{f2}
\delta^{(k)} \, :\, J^k(V) \, \longrightarrow\, J^k(W)
\end{equation}
for every $k\, \geq\, 0$.

Consider the homomorphism $q^0_V$ in \eqref{f1}. Let
$$
(q^0_V)^{(1)}\,:\, J^1(J^1(V)) \, \longrightarrow\, J^1(V)
$$
be the corresponding homomorphism in \eqref{f2}. On the other hand, we have the
homomorphism
$$
q^0_{J^1(V)}\,:\, J^1(J^1(V)) \, \longrightarrow\, J^1(V)
$$
by setting $k\,=\, 0$ and $J^1(V)$ in place of $V$ in \eqref{f1}. Note that the following
two compositions
$$
J^1(J^1(V)) \, \stackrel{(q^0_V)^{(1)}}{\longrightarrow}\, J^1(V) 
\, \stackrel{q^0_V}{\longrightarrow}\, V
$$
and 
$$
J^1(J^1(V)) \, \stackrel{q^0_{J^1(V)}}{\longrightarrow}\, J^1(V) 
\, \stackrel{q^0_V}{\longrightarrow}\, V
$$
coincide. The kernel of the homomorphism $(q^0_V)^{(1)} - q^0_{J^1(V)}$ coincides with $J^2(V)$. Therefore,
we have the short exact sequence of holomorphic vector bundles
\begin{equation}\label{f3}
0\, \longrightarrow\, J^2(V) \, \stackrel{\mathbf t}{\longrightarrow}\, J^1(J^1(V)) \,
\stackrel{(q^0_V)^{(1)}-q^0_{J^1(V)}}{\longrightarrow}\, V\otimes K_X
\, \longrightarrow\, 0
\end{equation}
on $X$.

The sheaf of {\it holomorphic differential operators} of order $k$ from $V$ to another holomorphic vector 
bundle $W$ is the sheaf of holomorphic sections of the holomorphic vector bundle
\begin{equation}\label{f4}
W\otimes J^k(V)^*\,=\, \text{Hom}(J^k(V), \, W)\, =:\, \text{Diff}^k(V,\, W)\, .
\end{equation}
Using the inclusion $V\otimes K^{k}_X \, \hookrightarrow\, J^{k}(V)$ in \eqref{f1},
we have a surjective homomorphism of vector bundles
\begin{equation}\label{f5}
\text{Diff}^k(V,\, W)\,\longrightarrow\, \text{Hom}(V\otimes K^{k}_X, \, W)\,=\,
\text{Hom}(V, \, W)\otimes T_X^{\otimes k}\, ;
\end{equation}
it is known as the {\it symbol} map.

\section{A rank two parabolic bundle}\label{se3}

Let $X$ be a compact connected Riemann surface. The genus of $X$ will be
denoted by $g$. Fix a finite nonempty subset
\begin{equation}\label{e1}
S\,:=\, \{x_1,\, \cdots,\, x_m\}\, \subset\, X\, ,
\end{equation}
(so $m\, \geq\, 1$), and also fix a function
\begin{equation}\label{e2}
{\mathbf c} \, :\, S\, \longrightarrow\, \{t\, \in\, {\mathbb R}\, \mid\, t\, >\, 1\}\, ;
\end{equation}
for notational convenience, the real number ${\mathbf c}(x_i)$ will also be denoted by $c_i$.

We shall assume that $m\,=\, \# S$ is an even integer. Fix a pair
\begin{equation}\label{e3}
({\mathbb L},\, \varphi_0)\, ,
\end{equation}
where $\mathbb L$ is a holomorphic line bundle on $X$ such that ${\mathbb L}^{^{\otimes 2}}$
is isomorphic to ${\mathcal O}_X(-S)$, and
$$
\varphi_0\, :\, {\mathbb L}^{^{\otimes 2}}\, \longrightarrow\, {\mathcal O}_X(-S)
$$
is a holomorphic isomorphism of line bundles.

Let $(K^{1/2}_X,\, I_X)$ be a theta characteristic on $X$; this means that $K^{1/2}_X$ is a
holomorphic line bundle on $X$ of degree $g-1$, and
$$
I_X\, :\, (K^{1/2}_X)^{\otimes 2}\, \longrightarrow\, K_X
$$
is a holomorphic isomorphism of holomorphic line bundles. We shall 
identify $(K^{1/2}_X)^{\otimes (2i +j)}$ with $K^{\otimes i}_X\otimes (K^{1/2}_X)^{\otimes j}$ 
using $I^{\otimes i}_X\otimes \text{Id}_{(K^{1/2}_X)^{\otimes j}}$.
The line bundle $(K^{1/2}_X)^{\otimes i}$ will be denoted by $K^{i/2}_X$.
For notational convenience, the above isomorphism
$I_X$ will be suppressed, and a theta characteristic will be denoted simply by
$K^{1/2}_X$. The isomorphism $I_X$ will be used without mentioning it.

Using Serre duality, we have
$$
H^1(X,\, \text{Hom}(K^{-1/2}_X\otimes {\mathbb L},\,
K^{1/2}_X\otimes {\mathbb L}))\,=\, H^1(X,\, K_X)\,=\, H^0(X,\, {\mathcal O}_X)^*
\,=\, \mathbb C
$$
(the line bundle $\mathbb L$ is as in \eqref{e3}).
Hence $$1\, \in\, H^1(X,\, \text{Hom}(K^{-1/2}_X\otimes {\mathbb L},\,
K^{1/2}_X\otimes {\mathbb L}))\,=\,\mathbb C$$ corresponds to a nontrivial extension
\begin{equation}\label{e4}
0\, \longrightarrow\, K^{1/2}_X\otimes {\mathbb L}\, \longrightarrow\, E
\, \stackrel{p_{_0}}{\longrightarrow}\, K^{-1/2}_X\otimes {\mathbb L} \, \longrightarrow\,0\, .
\end{equation}
So the short exact sequence in \eqref{e4} does not split holomorphically.

The holomorphic vector bundle $E$ constructed in \eqref{e4} will be equipped with a parabolic
structure; see \cite{MS}, \cite{MY} for parabolic bundles.

The parabolic divisor for the parabolic structure on $E$ is the subset $S$ in \eqref{e1}.
For any $x_i\, \in\, S$, the quasiparabolic filtration of $E_{x_i}$ is
\begin{equation}\label{e5}
0\, \subset\, (K^{1/2}_X\otimes {\mathbb L})_{x_i}\, \subset\, E_{x_i}\, ,
\end{equation}
where $(K^{1/2}_X\otimes {\mathbb L})_{x_i}$ is the fiber of the line subbundle
$K^{1/2}_X\otimes {\mathbb L}\, \hookrightarrow\, E$ in \eqref{e4}. The
parabolic weight of this line $(K^{1/2}_X\otimes {\mathbb L})_{x_i}$ is
$\frac{2c_i-1}{2c_i}$, while the parabolic weight of $E_{x_i}$ is
$\frac{1}{2c_i}$, where $c_i\,:=\, {\mathbf c}(x_i)$ with $\mathbf c$ being the
function in \eqref{e2}; note that $\frac{1}{2c_i}\, <\, \frac{2c_i-1}{2c_i}$. The
parabolic vector bundle thus obtained will be denoted by $E_*$.

The parabolic degree of the above parabolic bundle $E_*$ is:
\begin{equation}\label{e6}
\text{par-deg}(E_*)\,=\, \text{degree}(E)+\sum_{i=1}^m (\frac{2c_i-1}{2c_i}+\frac{1}{2c_i})
\,=\, \text{degree}({\mathbb L}^{\otimes 2}) + m\,=\, 0\, .
\end{equation}

A \textit{logarithmic connection} on a holomorphic vector bundle $F$ singular over $S$ is a
holomorphic differential operator
$$
D\, :\, F\, \longrightarrow\, F\otimes K_X\otimes{\mathcal O}_X(S) 
$$
of order one such that
\begin{equation}\label{li}
D(fs)\,=\, f\cdot D(s)+ s\otimes df\, ,
\end{equation}
where $s$ is any locally defined 
holomorphic section of $F$ and $f$ is any locally defined holomorphic function on $X$. So a
logarithmic connection $D$ on $F$ gives a ${\mathcal O}_X$--linear homomorphism
\begin{equation}\label{lj}
J^1(F)\, \longrightarrow\, F\otimes K_X\otimes{\mathcal O}_X(S)
\end{equation}
(see \eqref{f4}). The Leibniz condition in \eqref{li} implies that the symbol of
$D$ is $$\text{Id}_F\, \in\, H^0(X,\, \text{End}(F)\otimes {\mathcal O}_X(S))\, .$$
Recall that a logarithmic connection on a Riemann surface is flat (its curvature vanishes identically) because
the sheaf of holomorphic two forms on $X$ is the zero sheaf.

Note that for any $x_i\, \in\, S$, the fiber
$(K_X\otimes{\mathcal O}_X(S))_{x_i}$
of $K_X\otimes{\mathcal O}_X(S)$ over $x_i$ is identified with $\mathbb C$ using the Poincar\'e
adjunction formula; more precisely, for any holomorphic coordinate function $z$ defined around $x_i$,
with $z(x_i)\,=\, 0$, the evaluation of the local section $\frac{1}{z}dz$ of $K_X\otimes{\mathcal O}_X(S)$
at the point $x_i$ is independent of the choice of the coordinate function $z$. Consequently,
we have the isomorphism
$$
{\mathbb C}\, \longrightarrow\, (K_X\otimes{\mathcal O}_X(S))_{x_i}\, ,\ \
b\,\longmapsto\, b\cdot \frac{1}{z}dz\Big\vert_{z=x_i}\, .
$$

The composition
$$
F\, \stackrel{D}{\longrightarrow}\, F\otimes K_X\otimes{\mathcal O}_X(S)\,
\longrightarrow\, (F\otimes K_X\otimes{\mathcal O}_X(S))_{x_i}\,=\, F_{x_i}
$$
(recall that $(K_X\otimes{\mathcal O}_X(S))_{x_i}\,=\, \mathbb C$)
is ${\mathcal O}_X$--linear, so it produces an endomorphism of the fiber $F_{x_i}$.
This element of $\text{End}(F_{x_i})$ is called the residue of $D$ at $x_i$, and
it is denoted by $\text{Res}(D,x_i)$ (see \cite{De}).

A \textit{connection} on the parabolic bundle $E_*$ is a logarithmic connection $D$ on
$E$, singular over $S$, such that
\begin{enumerate}
\item for all $1\,\leq\, i\,\leq\, m$, the residue $\text{Res}(D,x_i)\, \in\, 
\text{End}(E_{x_i})$ preserves the line $(K^{1/2}_X \otimes {\mathbb L})_{x_i}$ in \eqref{e5}, 
and $\text{Res}(D,x_i)$ acts on $(K^{1/2}_X\otimes {\mathbb L})_{x_i}$ as multiplication by 
the parabolic weight $\frac{2c_i-1}{2c_i}$, and

\item the endomorphism of the quotient $E_{x_i}/(K^{1/2}_X\otimes {\mathbb L})_{x_i}\,=\,
(K^{-1/2}_X \otimes {\mathbb L})_{x_i}$ (see \eqref{e4}) induced by
$\text{Res}(D,x_i)$ coincides with multiplication by the parabolic weight $\frac{1}{2c_i}$.
\end{enumerate}
(See \cite[Section~2.2]{BL}.)

Note that for a logarithmic connection $D$ on $E$ defining a connection on $E_*$, the trace of
the residue of $D$ at each point of $S$ is $1$.

\begin{remark}
We observe that there exist logarithmic connections $D$ on $E$ with more general residue maps 
$\text{Res}(D,x_i)$. We need the connection $D$ to be compatible with the parabolic structure. 
\end{remark}

A connection $D$ on $E_*$ is called \textit{reducible} if there is a holomorphic line subbundle
$L'\, \subset\, E$ preserved by $D$, meaning
\begin{equation}\label{rc}
D(L')\, \subset\, L'\otimes K_X\otimes{\mathcal O}_X(S)\, .
\end{equation}
A connection on $E_*$ is called \textit{irreducible} if it is not reducible.

\begin{proposition}\label{prop1}
Assume that ${\rm genus}(X)\, =\, g\, \geq\, 1$. Then the following two hold:
\begin{enumerate}
\item The parabolic bundle $E_*$ admits a connection.

\item Any connection on $E_*$ is irreducible.
\end{enumerate}
\end{proposition}

\begin{proof}
We shall first prove (2). To prove (2) by
contradiction, let $D$ be a connection on $E_*$ which is reducible.
Let $L'\, \subset\, E$ be a holomorphic line subbundle such that \eqref{rc} holds. 
Let $L'_*\, \subset\, E_*$ be the parabolic line subbundle given by $L'$ equipped with
the parabolic structure induced from $E_*$. Since \eqref{rc} holds, in particular, $L'_*$
admits a connection, we have
\begin{equation}\label{pdq}
\text{par-deg}(L'_*)\,=\,0
\end{equation}
\cite[p.~598, Lemma 4.2]{BL}, \cite[p.~16, Theorem 3]{Oh}. 

Equip the quotient bundle $K^{-1/2}_X\otimes {\mathbb L}$ in \eqref{e4} with the parabolic 
structure induced by the parabolic structure of $E_*$. For this parabolic line bundle we have
\begin{equation}\label{pda}
\text{par-deg}(K^{-1/2}_X\otimes {\mathbb L})\,=\,
-g+1 - \frac{m}{2} + \sum_{i=1}^m \frac{1}{2c_i} \, <\, 0 \,=\, \text{par-deg}(L'_*)
\end{equation}
(see \eqref{pdq}).

Let $\Phi$ denote the composition
$$
L'\, \hookrightarrow\, E\, \stackrel{p_{_0}}{\longrightarrow}\, K^{-1/2}_X\otimes {\mathbb L}
$$
(see \eqref{e4}). This $\Phi$ is a homomorphism of parabolic line bundles. Indeed, if
the parabolic weight of $L'_*$ at a point $x_i\, \in\, S$ is strictly bigger than the
parabolic weight of $K^{-1/2}_X\otimes {\mathbb L}$ at $x_i$, then
$$
L'_{x_i}\, =\, (K^{1/2}_X\otimes {\mathbb L})_{x_i}\,\subset\, E_{x_i}\, ,
$$
and hence $\Phi(x_i)\,=\, 0$. Since $\Phi$ is homomorphism of parabolic line bundles, from 
\eqref{pda} we conclude that $\Phi\,=\, 0$; note that there is no nonzero parabolic homomorphism
from a parabolic line bundle of higher parabolic degree to a parabolic line bundle of
lower parabolic degree. Consequently, $L'$ coincides with the line subbundle $K^{1/2}_X\otimes
{\mathbb L}\, \subset\, E$ in \eqref{e4}.

On the other hand, the parabolic degree of the parabolic line subbundle
$K^{1/2}_X\otimes {\mathbb L}$ equipped with the induced parabolic structure is
$$
g-1 - \frac{m}{2} + \sum_{i=1}^m \frac{1-2c_i}{2c_i} \, > \, 0 \,=\, \text{par-deg}(L'_*)\, .
$$
Therefore, $L'$ can't coincide with $K^{1/2}_X\otimes {\mathbb L}$.
In view of this contradiction we conclude that the connection $D$ on $E_*$ is irreducible.

To prove (1) in the proposition, it suffices to show that the vector bundle $E$ is 
indecomposable. Indeed, any indecomposable parabolic vector bundle of parabolic degree zero 
admits a connection \cite[p.~599, Proposition 4.1]{BL}, and hence from \eqref{e6} it follows 
that $E_*$ admits a connection if $E$ is indecomposable.

We first assume that $\text{genus}(X)\, =\, g\, >\, 1$.

To prove that $E$ is indecomposable by contradiction, assume that
\begin{equation}\label{e7}
E\,=\, L\oplus M\, ,
\end{equation}
where $L$ and $M$ are holomorphic line bundles on $X$ with $\text{degree}(M)\, \geq\,
\text{degree}(L)$. Therefore, we have
\begin{equation}\label{e9}
\text{degree}(M)\, \geq\, \frac{\text{degree}(L)+\text{degree}(M)}{2}\,=\,
\frac{\text{degree}(E)}{2}\,= \,\text{degree}(K^{-1/2}_X\otimes {\mathbb L})+g-1\, .
\end{equation}
From this it follows that $H^0(X,\, \text{Hom}(M,\, K^{-1/2}_X\otimes {\mathbb L}))\,=\,0$,
because we have $\text{degree}(M)\,>\, \text{degree}(K^{-1/2}_X\otimes {\mathbb L})$. In particular,
the composition
\begin{equation}\label{e8}
M\, \hookrightarrow\, E \, \stackrel{p_{_0}}{\longrightarrow}\, K^{-1/2}_X\otimes {\mathbb L}
\end{equation}
vanishes identically (see \eqref{e4} (for $p_0$) and \eqref{e7}). Consequently, the subbundle 
$M$ of $E$ in \eqref{e7} coincides with the subbundle $K^{1/2}_X\otimes {\mathbb L}$ in 
\eqref{e4}. This implies that the decomposition in \eqref{e7} gives a holomorphic splitting of 
the short exact sequence in \eqref{e4}. But, as noted earlier, the short exact sequence in 
\eqref{e4} does not split holomorphically. So the vector bundle $E$ is indecomposable.

Next assume that $g\,=\,1$. We shall show that $E$ is indecomposable in this case as well. To 
prove this, assume, as before, that we have a holomorphic decomposition as in \eqref{e7}. If the 
composition in \eqref{e8} is the zero homomorphism, then the previous argument shows that $E$ is 
indecomposable. So assume that the composition in \eqref{e8} is not identically zero. Then we have
$$
\text{degree}(M)\, \leq\,\text{degree}(K^{-1/2}_X\otimes {\mathbb L})
\,= \,\text{degree}(K^{-1/2}_X\otimes {\mathbb L})+g-1\, ,
$$
so from \eqref{e9} it follows immediately that $\text{degree}(M)\, =\, 
\text{degree}(K^{-1/2}_X\otimes {\mathbb L})$. This implies that the composition in \eqref{e8} 
is actually an isomorphism. Hence $M$ is a direct summand of the subbundle $K^{1/2}_X\otimes 
{\mathbb L}\, \subset\, E$ in \eqref{e4}, and the short exact sequence in \eqref{e4} splits
holomorphically. Since the short exact sequence in \eqref{e4} does not split
holomorphically, we once again conclude that $E$ is indecomposable.
\end{proof}

\begin{remark}\label{rem1}
Proposition \ref{prop1}(1) is not valid for $g\,=\, 0$. To construct an example, take $m\,=\, 
2$, $c_1\,=\, 3$ and $c_2\,=\, 2$. Any holomorphic vector bundle on ${\mathbb C}{\mathbb P}^1$ 
holomorphically splits into a direct sum of holomorphic line bundles \cite[p.~122, 
Th\'eor\`eme 1.1]{Gr}. Using the fact that the exact sequence in \eqref{e4} does not split
holomorphically, it is straightforward to check that
the vector bundle $E$ is isomorphic to ${\mathbb L}\oplus {\mathbb L}$
with $\text{degree}({\mathbb L})\,=\, \text{degree}(E)/2\,=\, -1$. Indeed, if $E\,=\, L\oplus M$,
with $\text{degree}(M)\, >\, \text{degree}(L)$, then
$$
\text{degree}(K^{1/2}_X\otimes {\mathbb L})\,=\, -2\, <\,
\text{degree}(M)\, \geq\, \text{degree}(K^{-1/2}_X\otimes {\mathbb L})\,=\, 0\, ,
$$
and hence $M$ projects isomorphically to the quotient $K^{-1/2}_X\otimes {\mathbb L}$ in \eqref{e4}, making
$M$ a direct summand of the line subbundle $K^{1/2}_X\otimes {\mathbb L}$, and thus implying that
\eqref{e4} splits holomorphically. Therefore, we have
$E\,=\, {\mathbb L}\oplus {\mathbb L}$.

Take a direct summand ${\mathbb L}'\,=\, {\mathbb L}$ of $E$
such that ${\mathbb L}'_{x_1}\, \subset\, E_{x_1}$ 
coincides with $(K^{1/2}_X\otimes {\mathbb L})_{x_1}$ in \eqref{e4}. Note that the
two subbundles ${\mathbb L}'$ and $K^{1/2}_X\otimes {\mathbb L}$ of $E$ are distinct because
$\text{degree}({\mathbb L}')\, \not=\, \text{degree}(K^{1/2}_X\otimes {\mathbb L})$. Consider
the short exact sequence on $X$
$$
0\,\longrightarrow\, {\mathbb L}'\oplus (K^{1/2}_X\otimes {\mathbb L})\,\longrightarrow\, E
\,\longrightarrow\, Q \,\longrightarrow\, 0\, .
$$
Since $\text{degree}(E) - (\text{degree}({\mathbb L}') + \text{degree}(K^{1/2}_X\otimes
{\mathbb L}))\,=\, 1$, it follows that $Q$ is of degree one and hence it is supported on $x_1$;
note that $x_1$ is contained in the support of $Q$ by the condition on ${\mathbb L}'$
that ${\mathbb L}'_{x_1}$ coincides with $(K^{1/2}_X\otimes {\mathbb L})_{x_1}$.
In particular, the fiber ${\mathbb L}'_{x_2}\, \subset\, E_{x_2}$ does not coincide with
$(K^{1/2}_X\otimes {\mathbb L})_{x_2}$. Therefore,
the parabolic degree of this line subbundle ${\mathbb L}'
\,\subset\, E$, equipped with the induced parabolic structure, is
$$
-1+ \frac{5}{6} + \frac{1}{4}\,=\, \frac{1}{12} \, \not=\, 0\, .
$$

This parabolic line subbundle ${\mathbb L}' \,\subset\, E$ has a parabolic direct summand 
given by the copy of $\mathbb L$ whose fiber over $x_2$ coincides with $(K^{1/2}_X\otimes 
{\mathbb L})_{x_2}$. Hence the parabolic bundle $E_*$ does not admit a connection 
\cite[p.~601, Corollary 5.1]{BL}.
\end{remark}

Henceforth, we shall always assume that $g\, \geq\, 1$.

Let
\begin{equation}\label{eta}
\eta\, :\, K_X\,=\,\text{Hom}(K^{-1/2}_X\otimes {\mathbb L},\,
K^{1/2}_X\otimes {\mathbb L}) \, \longrightarrow\, \text{End}(E)
\end{equation}
be the homomorphism that sends any $w\, \in\, \text{Hom}(K^{-1/2}_X\otimes {\mathbb L},\,
K^{1/2}_X\otimes {\mathbb L})_x$ to the composition
$$
E_x \, \stackrel{p_{_0}(x)}{\longrightarrow} \, (K^{-1/2}_X\otimes {\mathbb L})_x\,
\stackrel{w}{\longrightarrow}\,(K^{1/2}_X\otimes {\mathbb L})_x
\, \hookrightarrow\, E_x
$$
(see \eqref{e4}). It is easy to check that $\eta(K_X)$ is a holomorphic line subbundle of
$\text{End}(E)$. Using the injective map
$$
H^0(X,\, K_X)\, \longrightarrow\, H^0(X,\, \text{End}(E))\, , \ \
\alpha\, \longmapsto\, \eta(\alpha)\, ,
$$
$H^0(X,\, K_X)$ will be considered as a subspace of $H^0(X,\, \text{End}(E))$. This subspace
\begin{equation}\label{sp}
H^0(X,\, K_X)\,\subset\, H^0(X,\, \text{End}(E))
\end{equation}
evidently consists of only nilpotent endomorphisms of $E$.

\begin{lemma}\label{lem0}
Any endomorphism $T$ of the holomorphic vector bundle $E$ is of the form
$$
T\,=\, c\cdot {\rm Id}_E + \alpha\, ,
$$
where $c\, \in\, \mathbb C$ and $\alpha\,\in\, H^0(X,\, K_X)$. Hence any holomorphic 
automorphism of $E$ is of the form $c\cdot {\rm Id}_E + \alpha$ with $c\, \in\, {\mathbb 
C}\setminus \{0\}$ and $\alpha\,\in\, H^0(X,\, K_X)$.
\end{lemma}

\begin{proof}
In the proof of Proposition \ref{prop1}(2) we saw that
the vector bundle $E$ is indecomposable. Therefore,
any element $T\, \in\, H^0(X,\, \text{End}(E))$ is of the form
$c\cdot {\rm Id}_E + \alpha'$, where $c\, \in\, \mathbb C$ and $\alpha'$ is
nilpotent \cite[p.~201, Proposition 15]{At2} (see also \cite{At1}). Take a nonzero
nilpotent endomorphism $$\alpha'\ \in\, H^0(X,\, \text{End}(E))\, ,$$
so $\alpha'\circ\alpha'\,=\, 0$. Let $F\, \subset\, E$
be the holomorphic line subbundle generated by $\text{kernel}(\alpha')$; so $F$ is the
inverse image, in $E$, of the torsion part of $E/\text{kernel}(\alpha')$. We note that
the image of $\alpha'$ also generates $F$. To prove the lemma it suffices to show that
the composition
\begin{equation}\label{leq1}
F \, \hookrightarrow\, E \,\stackrel{p_{_0}}{\longrightarrow} \,
K^{-1/2}_X\otimes {\mathbb L}
\end{equation}
vanishes identically, where $p_0$ is the projection in \eqref{e4}.

To prove by contradiction that the composition in \eqref{leq1} vanishes identically,
assume that the composition
in \eqref{leq1} does not vanish identically. Then we have the short exact sequence of sheaves
\begin{equation}\label{leq2}
0\,\longrightarrow\, F\oplus K^{1/2}_X\otimes {\mathbb L}\,\longrightarrow\, E
\,\longrightarrow\, Q \,\longrightarrow\, 0
\end{equation}
on $X$ which is constructed using the inclusions of $F$ and $K^{1/2}_X\otimes
{\mathbb L}$ in $E$. Note that $Q$ in \eqref{leq2} is a torsion sheaf. So we have
\begin{equation}\label{le3}
\text{degree}(F) + \text{degree}( K^{1/2}_X\otimes {\mathbb L})+\text{degree}(Q)\,=\,
\text{degree}(E)\,=\, -m\, .
\end{equation}
Since $F$ contains a nonzero quotient of $E$, namely $\alpha'(E)$, as a subsheaf, it
can be deduced that
\begin{equation}\label{de}
\text{degree}(F)\, \geq \, \text{degree}(K^{-1/2}_X\otimes {\mathbb L})\, ;
\end{equation}
indeed, if $\text{degree}(F)\, < \, \text{degree}(K^{-1/2}_X\otimes {\mathbb L})$, then
none of the two line bundles $K^{1/2}_X\otimes {\mathbb L}$ and $K^{-1/2}_X\otimes {\mathbb L}$
has a nonzero homomorphism to $F$, in which case $F$ can't contain a nonzero quotient of $E$.

{}From \eqref{de} we have
$$
\text{degree}(F) + \text{degree}( K^{1/2}_X\otimes {\mathbb L})\, \geq\,
\text{degree}(K^{-1/2}_X\otimes {\mathbb L}) + \text{degree}(K^{1/2}_X\otimes {\mathbb L})\,
=\, \text{degree}(E)\, .
$$
Hence from \eqref{le3} it now follows that $\text{degree}(Q)\,=\,0$. So we have $Q\,=\, 0$,
because $Q$ is a torsion sheaf of degree zero. This 
implies that $F$ gives a holomorphic splitting of the short exact sequence in \eqref{e4}.

Since \eqref{e4} does not split holomorphically, we conclude that the composition in
\eqref{leq1} vanishes identically. As noted before, the lemma follows from this.
\end{proof}

\begin{corollary}\label{cor2}
Any endomorphism of the holomorphic vector bundle $E$ preserves the parabolic structure of $E_*$.
\end{corollary}

\begin{proof}
This follows immediately from Lemma \ref{lem0}.
\end{proof}

Let
$$
D\, :\, E\, \longrightarrow\, E\otimes K_X\otimes {\mathcal O}_X(S)
$$
be a connection on $E_*$. Consider the line subbundle $K^{1/2}_X\otimes {\mathbb L}$
of $E$ in \eqref{e4}. The composition homomorphism
\begin{equation}\label{ch}
K^{1/2}_X\otimes {\mathbb L}\, \hookrightarrow\, E\,\stackrel{D}{\longrightarrow}\,
E\otimes (K_X\otimes {\mathcal O}_X(S)) \,\stackrel{p_{_0}\otimes{\rm Id}}{\longrightarrow}\,
K^{-1/2}_X\otimes {\mathbb L}\otimes (K_X\otimes {\mathcal O}_X(S))
\,=\, K^{1/2}_X\otimes{\mathbb L}^{-1}\, ,
\end{equation}
where $p_0$ is the projection in \eqref{e4}, is known as the second fundamental form of the line
subbundle $K^{1/2}_X\otimes {\mathbb L}$ for the logarithmic connection $D$; here
${\mathbb L}\otimes {\mathcal O}_X(S)$ is identified with ${\mathbb L}^{-1}$ using
$\varphi_0$ in \eqref{e3}. From \eqref{li}
it follows immediately that the composition homomorphism
in \eqref{ch} is ${\mathcal O}_X$--linear. Let
\begin{equation}\label{ch2}
{\mathbb S}_D(K^{1/2}_X\otimes {\mathbb L})\, \in\, H^0(X,\,{\rm Hom}(K^{1/2}_X\otimes {\mathbb L},
\, K^{1/2}_X\otimes {\mathbb L}^{-1}))\,=\,
H^0(X,\,{\mathbb L}^{-2})\,=\,H^0(X,\,{\mathcal O}_X(S))
\end{equation}
be the second fundamental form of $K^{1/2}_X\otimes {\mathbb L}$ for the logarithmic connection $D$.

We recall from the definition of a connection on $E_*$ that the residue of $D$ at
any $x_i\, \in\, S$ preserves the line 
$(K^{1/2}_X\otimes {\mathbb L})_{x_i}\, \subset\, E_{x_i}$. From this it follows immediately
that the section ${\mathbb S}_D(K^{1/2}_X\otimes {\mathbb L})$ in \eqref{ch2} vanishes at all
$x_i\,\in\, S$. Consequently, we have
\begin{equation}\label{e10}
{\mathbb S}_D(K^{1/2}_X\otimes {\mathbb L})\, \in\, H^0(X,\,{\mathcal O}_X)
\, \subset\, H^0(X,\,{\mathcal O}_X(S))\, .
\end{equation}

\begin{lemma}\label{lem1}
The holomorphic function ${\mathbb S}_D(K^{1/2}_X\otimes {\mathbb L})$ in \eqref{e10}
does not vanish identically.
\end{lemma}

\begin{proof}
Assume that ${\mathbb S}_D(K^{1/2}_X\otimes {\mathbb L})\,=\, 0$. This implies that the 
connection $D$ preserves the line subbundle $K^{1/2}_X\otimes {\mathbb L}$ inducing a 
logarithmic connection on it. Let $D'$ be the logarithmic connection on $K^{1/2}_X\otimes 
{\mathbb L}$ induced by $D$. We recall the general formula relating the residue of a logarithmic 
connection with the degree of the vector bundle:
\begin{equation}\label{a1}
\text{degree}(K^{1/2}_X\otimes {\mathbb L})+ \sum_{i=1}^m\text{trace}(\text{Res}(D', x_i))\,=\, 0\, ,
\end{equation}
where $\text{Res}(D', x_i)$ is the residue of the logarithmic connection
$D'$ at $x_i$ \cite[p.~16, Theorem 3]{Oh}. Since
$\text{degree}(K^{1/2}_X\otimes {\mathbb L})\,=\, g-1-\frac{m}{2}$ and
$\text{Res}(D', x_i)\,=\, \frac{2c_i-1}{2c_i}$, it follows that
$$
\text{degree}(K^{1/2}_X\otimes {\mathbb L})+ \sum_{i=1}^m\text{trace}(\text{Res}(D', x_i))\,>\, 0\, ;
$$
recall that $g\, \geq\,1$ and $m\, \geq\, 1$. Since this contradicts \eqref{a1},
we conclude that ${\mathbb S}_D(K^{1/2}_X\otimes {\mathbb L})\,\not=\, 0$.
\end{proof}

The following is an immediate consequence of Lemma \ref{lem1}.

\begin{corollary}\label{cor1}
Let $D$ be a connection on $E_*$. The section 
${\mathbb S}_D(K^{1/2}_X\otimes {\mathbb L})$ in \eqref{e10}
is given by a nonzero constant function on $X$. In other words, ${\mathbb S}_D(K^{1/2}_X\otimes
{\mathbb L})$, considered as a section of ${\mathcal O}_X$, does not vanish anywhere on $X$, and
${\mathbb S}_D(K^{1/2}_X\otimes {\mathbb L})$, considered as a section of ${\mathcal O}_X(S)$,
vanishes exactly on $S$.
\end{corollary}

\section{Parabolic ${\rm SL}_2$--opers}\label{se4}

Recall that a logarithmic connection on $E$ induces a logarithmic connection on $\bigwedge^2 E$.
Note also  that $\det E \,=\, \bigwedge^2 E\,=\, {\mathbb L}^{^{\otimes 2}}\,=\, {\mathcal O}_X(-S)$.
The de Rham differential $f\, \longmapsto\, df$ produces a logarithmic connection on
${\mathcal O}_X(-S)$. The residue of this logarithmic connection on ${\mathcal O}_X(-S)$
at every point $x_i\, \in\, S$ is $1$.

Two connections on $E_*$ are called \textit{equivalent} is they are conjugate by a holomorphic 
automorphism of the parabolic vector bundle $E_*$. From Corollary \ref{cor2} we know that any 
holomorphic automorphism of $E$ is an automorphism of the parabolic vector bundle $E_*$. 
Consequently, two connections on $E_*$ are equivalent if they are conjugate by a holomorphic 
automorphism of $E$. Note that if a connection $D$ on $E_*$ has the property that the 
logarithmic connection on $\bigwedge^2 E$ induced by $D$ coincides with the tautological 
logarithmic connection on ${\mathcal O}_X(-S)$ given by the de Rham differential, then any 
connection on $E_*$ equivalent to $D$ also has this property. Indeed, if $D'$ and $D''$ are 
two equivalent connections on $E_*$ differing by a holomorphic automorphism $T$ of $E$, then 
the two logarithmic connections on $\det E \,=\,{\mathcal O}_X(-S)$ induced by $D'$ and $D''$ 
differ by the automorphism of $\det E$ induced by $T$. On the other hand, the automorphisms of 
the holomorphic line bundle $\det E$ are constant scalar multiplications, and a constant 
scalar multiplication preserves any logarithmic connection.

\begin{definition}\label{def1}
A {\it parabolic} ${\rm SL}_2$--oper is an equivalence class of
connections $D$ on $E_*$ such that the
logarithmic connection on $\bigwedge^2 E$ induced by $D$ coincides with the tautological
logarithmic connection on $\det E \,=\, {\mathcal O}_X(-S)$ given by the de Rham differential.
\end{definition}

\begin{remark}
Note that our definition of ${\rm SL}_2$--oper (in the case $S =
\emptyset$) is slightly more restrictive than the original one (see \cite{BD} section 2.8), as we choose a theta-characteristic $K^{1/2}_X$, which completely determines the underlying vector bundle.
\end{remark}

\begin{lemma}\label{lem2}
The space of parabolic ${\rm SL}_2$--opers on $X$ is nonempty.
\end{lemma}

\begin{proof}
Let $D$ be a connection on $E_*$, which exists by Proposition \ref{prop1}(1).
Let $D'$ be the logarithmic connection on 
$\bigwedge^2 E\,=\, {\mathcal O}_X(-S)$ induced by $D$. For any $x_i\, \in\, S$,
since the eigenvalues of ${\rm Res}(D, x_i)$ are $\frac{2c_i-1}{2c_i}$ and
$\frac{1}{2c_i}$, we conclude that ${\rm Res}(D', x_i)\,=\, \text{trace}({\rm Res}(D, x_i))\,=\, 1$.
In particular, ${\rm Res}(D', x_i)$ coincides with the residue at $x_i$ of the tautological
logarithmic connection $D_0$ on ${\mathcal O}_X(-S)$ given by the de Rham differential. So we have
$$
D'\, =\, D_0 +\beta\, ,$$
where $\beta\,\in\, H^0(X, \, K_X)$. Now it is straight-forward to check
that the logarithmic connection $D- \frac{\beta}{2}$ on $E$ defines a
parabolic ${\rm SL}_2$--oper.
\end{proof}

\begin{lemma}\label{lem3}
Let $D$ be a connection on $E_*$ defining a parabolic ${\rm SL}_2$--oper on $X$. Then $D$ 
produces a holomorphic isomorphism of $E$ with the jet bundle $J^1(K^{-1/2}_X\otimes {\mathbb 
L})$. This isomorphism takes the subbundle $K^{1/2}_X\otimes {\mathbb L}\, \subset\, E$ in 
\eqref{e4} to the subbundle $$(K^{-1/2}_X\otimes {\mathbb L})\otimes K_X\,=\, K^{1/2}_X\otimes 
{\mathbb L}\, \subset\, J^1(K^{-1/2}_X\otimes {\mathbb L})$$ in \eqref{f1}.
\end{lemma}

\begin{proof}
Consider the ${\mathcal O}_X$--linear homomorphism
$$
D'\, :\, J^1(E)\, \longrightarrow\, E\otimes K_X\otimes{\mathcal O}_X(S)
$$
given by $D$ (see \eqref{lj}). We shall determine $\text{kernel}(D')\,\subset\,
J^1(E)$. It can be shown that on the complement $X\setminus S$, we have
$\text{kernel}(D')\vert_{X\setminus S}\, =\,E\vert_{X\setminus S}$. Indeed,
$D$ is a holomorphic connection on $E_{X\setminus S}$, and, as mentioned before,
any holomorphic connection on a Riemann surface is flat. The above map
$E\vert_{X\setminus S}\,\longrightarrow\, \text{kernel}(D')\vert_{X\setminus S}$ sends
any $v\, \in\, E_x$ to the element of $J^1(E)_x$ defined by the unique flat section $s$ for
$D$, defined around $x$, with $s(x)\,=\, v$. This isomorphism
$E\vert_{X\setminus S}\,\longrightarrow\, \text{kernel}(D')\vert_{X\setminus S}$ clearly
extends to a homomorphism
\begin{equation}\label{idke}
D_K\, :\, E\otimes{\mathcal O}_X(-S)\, \longrightarrow\, \text{kernel}(D')\, \subset\,
J^1(E)
\end{equation}
over $X$. Let
$$
Q_D\, :=\, \text{kernel}(D')/(E\otimes{\mathcal O}_X(-S))
$$
be the quotient, which is a torsion sheaf; its support is contained in $S$, because
on $X\setminus S$ both $\text{kernel}(D')$ and $E\otimes{\mathcal O}_X(-S)$ are identified
with $E\vert_{X\setminus S}$. Note that we have
\begin{equation}\label{idn}
\text{degree}(\text{kernel}(D'))\,=\,\text{degree}
(E\otimes{\mathcal O}_X(-S))+ \text{degree}(Q_D)\, .
\end{equation}

The homomorphism $D'$ is surjective, because the residue of $D$ at each point $x_i\, \in\, S$
is an isomorphism. From this surjectivity it follows that
$$
\text{degree}(\text{kernel}(D'))\,=\, \text{degree}(J^1(E))
-\text{degree}(E\otimes K_X\otimes{\mathcal O}_X(S))\,=\, - 3m\, .
$$
Since $\text{degree}(E\otimes{\mathcal O}_X(-S))\,=\, - 3m$, from \eqref{idn} it follows that
$\text{degree}(Q_D)\,=\, 0$. Since $Q_D$ is a torsion sheaf of degree zero, we conclude that
$Q_D\,=\, 0$. This implies that the homomorphism $D_K$ in \eqref{idke}
is an isomorphism.

Let $p^{(1)}_0\, :\, J^1(E)\, \longrightarrow\, J^1(K^{-1/2}_X\otimes {\mathbb L})$ be the homomorphism
constructed as in \eqref{f2} for the projection $p_0$ in \eqref{e4}. The composition
$$
p^{(1)}_0\circ D_K\, :\, E\otimes{\mathcal O}_X(-S)\, \longrightarrow\,
J^1(K^{-1/2}_X\otimes {\mathbb L})\, ,
$$
where $D_K$ is the isomorphism in \eqref{idke}, clearly vanishes over $S$. So $p^{(1)}_0\circ D_K$
produces a homomorphism
\begin{equation}\label{idke2}
\widetilde{p^{(1)}_0\circ D_K}\, :\, E\, \longrightarrow\,
J^1(K^{-1/2}_X\otimes {\mathbb L})\, .
\end{equation}
From Corollary \ref{cor1} we know that the section ${\mathbb S}_D(K^{1/2}_X\otimes {\mathbb L})
\, \in\, H^0(X,\,{\mathcal O}_X(S))$ in \eqref{e10}
does not vanish on $X\setminus S$. Using this it follows that $\widetilde{p^{(1)}_0\circ D_K}$
in \eqref{idke2} is an isomorphism over $X\setminus S$. Therefore, $J^1(K^{-1/2}_X\otimes {\mathbb L})/E$
is a torsion sheaf of degree zero, because
$\text{degree}(E)\,=\,-m\, =\, \text{degree}(J^1(K^{-1/2}_X\otimes {\mathbb L}))$. Hence we have
$J^1(K^{-1/2}_X\otimes {\mathbb L})/E\,=\, 0$, implying that
$\widetilde{p^{(1)}_0\circ D_K}$ in \eqref{idke2} is an isomorphism over entire $X$.

Over $X\setminus S$, the isomorphism $\widetilde{p^{(1)}_0\circ D_K}$ evidently
takes the subbundle $$(K^{1/2}_X\otimes {\mathbb L})\vert_{X\setminus S}\, \subset\,
E\vert_{X\setminus S}$$ in \eqref{e4} to the subbundle
$$((K^{-1/2}_X\otimes {\mathbb L})\otimes K_X)\vert_{X\setminus S}\,=\,
(K^{1/2}_X\otimes {\mathbb L})\vert_{X\setminus S}\, \subset\, J^1(K^{-1/2}_X\otimes
{\mathbb L})\vert_{X\setminus S}$$
in \eqref{f1}. Consequently, the isomorphism $\widetilde{p^{(1)}_0\circ D_K}$ over $X$
takes the subbundle $K^{1/2}_X\otimes {\mathbb L}\, \subset\, E$ to the subbundle
$K^{1/2}_X\otimes {\mathbb L}\, \subset\, J^1(K^{1/2}_X\otimes {\mathbb L})$.
\end{proof}

\begin{remark}\label{rem2}
It should be clarified that the isomorphism $\widetilde{p^{(1)}_0\circ D_K}$ in \eqref{idke2} depends
on the logarithmic connection $D$.
\end{remark}

\begin{theorem}\label{thm1}
The space of parabolic ${\rm SL}_2$--opers on $X$ is an affine space for
the vector space $H^0(X, \, K^2_X\otimes {\mathcal O}_X(S))$.
\end{theorem}

\begin{proof}
Let $D$ be a connection on $E_*$ defining a parabolic ${\rm SL}_2$--oper on $X$. Consider
the isomorphism $\widetilde{p^{(1)}_0\circ D_K}$ constructed in \eqref{idke2}. Let $$D_1\,=\,
(\widetilde{p^{(1)}_0\circ D_K})_* D$$ be the logarithmic connection on $J^1(K^{-1/2}_X
\otimes {\mathbb L})$ given by the logarithmic connection $D$
on $E$ using this isomorphism. The composition homomorphism
\begin{equation}\label{e23}
J^1(J^1(K^{-1/2}_X\otimes {\mathbb L}))\, \stackrel{D_1}{\longrightarrow}\,
J^1(K^{-1/2}_X\otimes {\mathbb L})\otimes K_X\otimes {\mathcal O}_X(S)\,\longrightarrow\,
J^0(K^{-1/2}_X\otimes {\mathbb L})\otimes K_X\otimes {\mathcal O}_X(S)
\end{equation}
$$
=\,(K^{-1/2}_X\otimes{\mathbb L})\otimes K_X\otimes{\mathcal O}_X(S)
\,=\,K^{1/2}_X\otimes{\mathbb L}\otimes{\mathcal O}_X(S)\,=\, K^{1/2}_X\otimes{\mathbb L}^*
$$
will be denoted by $D'_1$ (see \eqref{lj}); the projection
$$
J^1(K^{-1/2}_X\otimes {\mathbb L})\otimes K_X\otimes{\mathcal O}_X(S)\,\longrightarrow\,
J^0(K^{-1/2}_X\otimes {\mathbb L})\otimes K_X\otimes{\mathcal O}_X(S)
$$
in \eqref{e23} is the homomorphism $J^1(K^{-1/2}_X\otimes {\mathbb L})\,\longrightarrow\,
J^0(K^{-1/2}_X\otimes {\mathbb L})$ in \eqref{f1} tensored with the identity map
of $K_X\otimes{\mathcal O}_X(S)$. Consider the homomorphism $\mathbf t$ in \eqref{f3}. Let
$D''_1$ denote the composition
$$
J^2(K^{-1/2}_X\otimes {\mathbb L}) \, \stackrel{\mathbf t}{\longrightarrow}\,
J^1(J^1(K^{-1/2}_X\otimes {\mathbb L})) \, \stackrel{D'_1}{\longrightarrow}\,
K^{1/2}_X\otimes {\mathbb L}\otimes {\mathcal O}_X(S)\,=\, K^{1/2}_X\otimes {\mathbb L}^*\, .
$$
From the construction of the isomorphism $\widetilde{p^{(1)}_0\circ D_K}$ in \eqref{idke2}
it follows that this homomorphism $D''_1$ vanishes identically. Therefore, the image of
the composition $D_1\circ{\mathbf t}$ lies in the subbundle
$$
(K^{-1/2}_X\otimes {\mathbb L})\otimes K_X \otimes (K_X\otimes {\mathcal O}_X(S))\,\subset\,
J^1(K^{-1/2}_X\otimes {\mathbb L})\otimes (K_X\otimes {\mathcal O}_X(S))\, ;
$$
see \eqref{f1} for the above inclusion $(K^{-1/2}_X\otimes {\mathbb L})\otimes K_X\, \hookrightarrow\,
J^1(K^{-1/2}_X\otimes {\mathbb L})$. In other words, we have the homomorphism
$$
J^2(K^{-1/2}_X\otimes {\mathbb L}) \, \stackrel{D_1\circ{\mathbf t}}{\longrightarrow}\,
(K^{-1/2}_X\otimes {\mathbb L})\otimes K_X \otimes K_X\otimes {\mathcal O}_X(S)
$$
$$
=\, K^{3/2}_X\otimes {\mathbb L}\otimes {\mathcal O}_X(S)\,=\, K^{3/2}_X\otimes {\mathbb L}^*\, \subset\,
J^1(K^{-1/2}_X\otimes {\mathbb L})\otimes K_X\otimes {\mathcal O}_X(S)\, .
$$
Consequently, we have a holomorphic differential operator
\begin{equation}\label{e21}
D_1\circ{\mathbf t}\, \in\, H^0(X,\, \text{Diff}^2(K^{-1/2}_X\otimes {\mathbb L},\,
K^{3/2}_X\otimes {\mathbb L}\otimes {\mathcal O}_X(S)))\,=\,
H^0(X,\, \text{Diff}^2(K^{-1/2}_X\otimes {\mathbb L},\,
K^{3/2}_X\otimes {\mathbb L}^*))
\end{equation}
of order two (see \eqref{f4}).

While $D_1\circ{\mathbf t}$ is constructed above from the homomorphism $D_1$ in \eqref{e23}, 
we shall now show that the homomorphism $D_1$ can also be recovered back from the differential 
operator $D_1\circ{\mathbf t}$. For this, consider the two subsheaves
$$
{\mathbf t}(\text{kernel}(D_1\circ{\mathbf t}))\ \ \text{ and }\ \
J^1(K^{-1/2}_X\otimes {\mathbb L})\otimes K_X
$$
of $J^1(J^1(K^{-1/2}_X\otimes {\mathbb L}))$ (see \eqref{f1} for the second subsheaf).
The composition
$$
J^1(K^{-1/2}_X\otimes {\mathbb L})\otimes K_X\, \hookrightarrow\,
J^1(J^1(K^{-1/2}_X\otimes {\mathbb L}))\, \longrightarrow\,
J^1(J^1(K^{-1/2}_X\otimes {\mathbb L}))/({\mathbf t}(\text{kernel}(D_1\circ{\mathbf t})))
$$
vanishes on $S$. Therefore, this composition gives a homomorphism
\begin{equation}\label{e22}
D_2\, :\, J^1(K^{-1/2}_X\otimes {\mathbb L})\otimes K_X\otimes{\mathcal O}_X(S)\, \longrightarrow\,
J^1(J^1(K^{-1/2}_X\otimes {\mathbb L}))/({\mathbf t}(\text{kernel}(D_1\circ{\mathbf t})))
\end{equation}
which is an isomorphism over $X\setminus S$. On the other hand, we have
$$
\text{degree}(J^1(K^{-1/2}_X\otimes {\mathbb L})\otimes K_X\otimes{\mathcal O}_X(S))
\,=\, \text{degree}(J^1(J^1(K^{-1/2}_X\otimes {\mathbb L}))/({\mathbf t}
(\text{kernel}(D_1\circ{\mathbf t}))))\, .
$$
Since any homomorphism between two holomorphic vector bundles of same degree is an isomorphism
if it is generically an isomorphism, we now conclude that
the homomorphism $D_2$ in \eqref{e22} is an isomorphism. Finally, the composition
$$
J^1(J^1(K^{-1/2}_X\otimes {\mathbb L}))\, \longrightarrow\,
J^1(J^1(K^{-1/2}_X\otimes {\mathbb L}))/({\mathbf t}(\text{kernel}(D_1\circ{\mathbf t})))
$$
$$
\stackrel{D^{-1}_2}{\longrightarrow}\, 
J^1(K^{-1/2}_X\otimes {\mathbb L})\otimes K_X\otimes{\mathcal O}_X(S)
$$
coincides with $D_1$ in \eqref{e23}. Therefore, $D_1$ is uniquely determined by
$D_1\circ{\mathbf t}$.

It should be clarified that this does not imply that the map $D\, \longmapsto\, D_1$ is
injective. In fact, this map is \textit{not} injective.

The symbol of the differential operator $D_1\circ{\mathbf t}$ in \eqref{e21} coincides with the
section of ${\mathcal O}_X(S)$ given by the constant function $1$.
The condition in Definition \ref{def1} that the connection on $\bigwedge^2 E\,=\, {\mathcal 
O}_X(-S)$ induced by $D$ coincides with the tautological logarithmic connection on ${\mathcal 
O}_X(-S)$ given by the de Rham differential, implies that the connection on $\bigwedge^2 
J^1(K^{-1/2}_X\otimes {\mathbb L})\,=\, {\mathcal O}_X(-S)$ induced by $D_1$ coincides with the 
tautological logarithmic connection on ${\mathcal O}_X(-S)$ given by the de Rham differential. 
This is equivalent to the condition that the first order part of the differential operator 
$D_1\circ{\mathbf t}$ in \eqref{e21} vanishes. The space of second order holomorphic differential 
operators satisfying the above conditions is an affine space for the vector space
$$
H^0(X,\, \text{Hom}(K^{-1/2}_X\otimes {\mathbb L},\,
K^{3/2}_X\otimes {\mathbb L}\otimes {\mathcal O}_X(S)))\,=\, H^0(X, \, K^2_X\otimes {\mathcal O}_X(S))\, .
$$
Using this it follows that the space of parabolic ${\rm SL}_2$--oper on $X$ is an affine space for
the vector space $H^0(X, \, K^2_X\otimes {\mathcal O}_X(S))$.
\end{proof}

\section{Parabolic ${\rm SL}_r$--opers}\label{se5}

\subsection{Parabolic tensor product and parabolic dual}

In \cite{MY} an equivalent formulation of the definition of parabolic bundles was given. We shall
recall it now.

Let $V\, \longrightarrow\, X$ be a holomorphic vector bundle.
Let
$$
V_{x_i} \,=:\, F_{i,1} \, \supsetneq\, \cdots\, \supsetneq\,
F_{i,j} \, \supsetneq\, \cdots\,
\supsetneq\, F_{i,a_i}\, \supsetneq\, F_{i,a_i+1} \,=\, 0
$$
be a quasiparabolic filtration
over each point $x_i$ of $S$. Fix parabolic weights
$$
0\, \leq\, \alpha_{i,1} \, <\, \cdots\, <\, \alpha_{i,j}
\, <\, \cdots\, <\, \alpha_{i,a_i}\, <\, 1
$$
associated to these quasiparabolic flags (\cite{MS}, \cite[p. 67]{Se}).

Now, for a point $x_i\, \in\, S$ and $t\, \in\,
[0 ,\,1]$, let
$$
V^{i,t}\, \subset\, V
$$
be the coherent subsheaf defined as follows: if $t\, \leq\, \alpha_{i,1}$, then
$$
V^{i,t}\, =\, V\, ,
$$
if $t\, >\, \alpha_{i,1}$, then $V^{i,t}$ is defined by the short exact sequence of sheaves
$$
0\, \longrightarrow\, V^{i,t}\, \longrightarrow\, V
\, \longrightarrow\, V/F_{i,j+1} \, \longrightarrow\, 0\, ,
$$
where $j\, \in\, [1,\, a_i]$ is the largest number such that 
$\alpha_{i,j}\, < \, t$. For $t\, \in\, [0,\, 1]$, define
$$
V^{(t)}\, =\, \bigcap_{i=1}^m V^{i,t}\, \subset\, V \, .
$$
Now we have a filtration of coherent sheaves
$\{V_t\}_{t\in \mathbb R}$ defined by
$$
V_t\, :=\, V^{(t-[t])}\otimes {\mathcal O}_X(-[t]S)\, ,
$$
where $[t]$ is the integral part of $t$, meaning $0\, \leq\, t-[t]\, <\, 1$. From
the construction of the filtration $\{V_t\}_{t\in \mathbb R}$
it is evident that the parabolic vector bundle $(V\, ,
\{F_{i,j}\}\, ,\{\alpha_{i,j}\})$ can be recovered from it. 
This description of parabolic bundles was introduced in \cite{MY}.

Let
$$
\iota\, :\, X\setminus S\, \hookrightarrow\, X
$$
be the inclusion of the complement. Let $\{V_t\}_{t\in \mathbb R}$ and $\{W_t\}_{t\in \mathbb R}$
be the filtrations corresponding to two parabolic
vector bundles $V_*$ and $W_*$ respectively. Consider the 
torsionfree quasi--coherent
sheaf $\iota_*((V_0\bigotimes W_0)\vert_{X\setminus S})$ on $X$.
Note that $V_s\bigotimes W_t$ is a coherent subsheaf of it for
all $s$ and $t$. For any $t\, \in\,
\mathbb R$, let
$$
{\mathcal E}_t\, \subset\, \iota_*((V_0\otimes W_0)\vert_{X
\setminus S})
$$
be the quasi--coherent subsheaf generated by all
$V_\alpha\bigotimes W_{t-\alpha}$, $\alpha\, \in\, \mathbb R$.
It is easy to see that ${\mathcal E}_t$ is a coherent sheaf, and
the collection $\{{\mathcal E}_t\}_{t\in \mathbb R}$ satisfies all
the three conditions needed to define a parabolic vector bundle
on $X$ with parabolic structure over $S$.
The parabolic vector bundle defined by $\{{\mathcal E}_t\}_{t\in \mathbb 
R}$ is denoted by $V_*\bigotimes W_*$, and it is called the
\textit{tensor product} of $V_*$ and $W_*$.

Now consider the torsionfree quasi--coherent
sheaf $\iota_*((V^*_0\bigotimes W_0)\vert_{X\setminus S})$ on $X$.
For any $t\, \in\, \mathbb R$, let
$$
{\mathcal F}_t\, \subset\, \iota_*((V^*_0\otimes W_0)\vert_{X
\setminus S})
$$
be the quasi--coherent subsheaf generated by all
$V^*_\alpha\bigotimes W_{\alpha+t}$, $\alpha\, \in\, \mathbb R$.
This ${\mathcal F}_t$ is a coherent sheaf, and
the collection $\{{\mathcal F}_t\}_{t\in \mathbb R}$ satisfies the
three conditions needed to define a parabolic vector bundle
with parabolic structure over $S$.
The parabolic vector bundle defined by $\{{\mathcal F}_t\}_{t\in
\mathbb R}$ is denoted by ${\rm Hom}(V_*\, , W_*)$. When $W_*$ is the
trivial parabolic line bundle with trivial parabolic structure, then
${\rm Hom}(V_*\, , W_*)$ is the parabolic dual $V^*_*$.

For any parabolic structure on $V$ and any integer $d\, \geq\, 2$, the tensor product $V^{\otimes d}$ is a subsheaf of the vector bundle underlying the
parabolic tensor product $(V_*)^{\otimes d}$. Let 
${\mathcal S}_d$ denote the
group of permutations of the set $\{1,\, \cdots,\, d\}$. The natural action
of ${\mathcal S}_d$ on $V^{\otimes d}$ that permutes the factors extends to
an action of ${\mathcal S}_d$ on the vector bundle underlying the
parabolic bundle $(V_*)^{\otimes d}$. The parabolic symmetric
product $\text{Sym}^d(V_*)$ is the parabolic bundle given by the fixed point locus of this
action of ${\mathcal S}_d$ on $(V_*)^d$; note that this fixed point locus
is equipped with the induced parabolic structure.

The exterior products of a parabolic vector bundle are defined in a similar way. The parabolic 
determinant bundle of a parabolic vector bundle $V_*$ is the $m$-th degree parabolic exterior 
product of $V_*$, where $m$ is the rank of $V_*$.

\subsection{Connections on a parabolic symmetric product}\label{se5.2}

Consider the rank two parabolic bundle $E_*$ in Section \ref{se3}. For any integer $r\, \geq\, 2$,
the parabolic symmetric product $\text{Sym}^{r-1}(E_*)$ will be denoted by ${\mathbb E}_r$. So
${\mathbb E}_r$ is a parabolic vector bundle of rank $r$.

\begin{corollary}\label{cor3}
The parabolic bundle ${\mathbb E}_r$ admits a connection.
\end{corollary}

\begin{proof}
A connection on $E_*$ induces a connection on the parabolic tensor product $(E_*)^{\otimes 
(r-1)}$. This connection on $(E_*)^{\otimes (r-1)}$ preserves the parabolic subbundle 
${\mathbb E}_r\,=\, \text{Sym}^{r-1}(E_*) \, \subset\, (E_*)^{\otimes (r-1)}$. Therefore, from 
Proposition \ref{prop1}(1) it follows that ${\mathbb E}_r$ admits a connection.
\end{proof}

Note that the parabolic determinant bundle $\det E_*$ for $E_*$, which is the same as the second parabolic
exterior product of $E_*$, is the trivial line bundle on $X$ with the trivial parabolic
structure. We observe that the determinant of the underlying
bundle $E$ equals ${\mathcal O}_X(-S)$ with parabolic weights
$1$ on $S$, which corresponds to the trivial parabolic line
bundle ${\mathcal O}_X$ with parabolic weights $0$ on $S$.
Using this it is straight-forward to deduce that parabolic determinant bundle
$\det {\mathbb E}_r$, which is the same as the $r$--th parabolic exterior product of ${\mathbb E}_r$,
is also the trivial line bundle on $X$ with the trivial parabolic structure.

A connection on the parabolic vector bundle ${\mathbb E}_r$ induces a connection on the
parabolic line bundle $\det {\mathbb E}_r$. Indeed, a connection on ${\mathbb E}_r$
induces a connection on the parabolic tensor product $({\mathbb E}_r)^{\otimes r}$. This
induced connection on $({\mathbb E}_r)^{\otimes r}$ preserves the parabolic subbundle
$\det {\mathbb E}_r\, \subset\, ({\mathbb E}_r)^{\otimes r}$.

In the proof of Corollary \ref{cor3} we saw that a connection $D$ on the parabolic vector bundle 
$E_*$ produces a connection on the parabolic vector bundle ${\mathbb E}_r$. Now if $D$ is a ${\rm 
SL}_2$--oper, meaning the logarithmic connection on $\bigwedge^2 E\,=\, {\mathcal O}_X(-S)$ 
induced by $D$ coincides with the one given by the de Rham differential, then the connection on the
parabolic determinant line bundle $\bigwedge^2 E_*$ coincides with the trivial connection on the 
trivial line bundle given by the de Rham differential. Using this it follows that the connection 
on the parabolic bundle ${\mathbb E}_r$ given by $D$ has the property that the connection on the 
parabolic line bundle $\det {\mathbb E}_r$ induced by it also coincides with the trivial 
connection on the trivial line bundle given by the de Rham differential.

Two connections on the parabolic bundle ${\mathbb E}_r$ are called \textit{equivalent}
is they differ by a holomorphic automorphism
of the parabolic vector bundle ${\mathbb E}_r$.

\begin{definition}\label{def2}
A {\it parabolic} ${\rm SL}_r$--{\it oper} is an equivalence class of
connections $D$ on the parabolic bundle ${\mathbb E}_r$ such that the
connection on $\det {\mathbb E}_r$ induced by $D$ coincides with the trivial
connection on $\det {\mathbb E}_r \,=\, {\mathcal O}_X$ given by the de Rham differential.
\end{definition}

\section{Special rational parabolic weights}\label{se6}

Henceforth, we assume that
$$
c_i\, :=\, {\mathbf c}(x_i)\, \in\, \mathbb N
$$
for all $x_i\, \in\, S$, where $\mathbf c$ is the function in \eqref{e2}. We also assume that
the integer $r$ is odd.

There is a ramified Galois covering
\begin{equation}\label{gc}
\rho\, :\, Y\, \longrightarrow\, X
\end{equation}
satisfying the following conditions:
\begin{itemize}
\item $\rho$ is unramified over the complement $X\setminus S$;

\item for every $x_i\, \in\, S$ and each point $y\, \in\, \rho^{-1}(x_i)$, the order of ramification
of $\rho$ at $y$ is $c_i$.
\end{itemize}
Such a covering $\rho$ exists; see \cite[p. 26, Proposition 1.2.12]{Na}.

The Galois group $\text{Gal}(\rho)$ for $\rho$ will
be denoted by $\Gamma$. The holomorphic cotangent bundle of $Y$ will be denoted by $K_Y$.
The action of $\Gamma$ on $Y$ produces an action of $\Gamma$ on $K_Y$.

A $\Gamma$--equivariant vector bundle on $Y$ is a holomorphic
vector bundle $V$ on $Y$ equipped with a lift of the action of $\Gamma$ such that the
action of any $\gamma\, \in\, \Gamma$ on $V$ maps $V_y$ to $V_{\gamma(y)}$ linearly
for every $y\, \in\, Y$.

Since the parabolic weights of $E_*$ at $x_i\, \in\, S$ are integral multiples of 
$\frac{1}{2c_i}$, the parabolic weights, at $x_i$, of the parabolic symmetric product 
$\text{Sym}^j(E_*)$ are integral multiples of $\frac{j}{2c_i}$. In particular, the parabolic 
weights of $\text{Sym}^{2j}(E_*)$ at $x_i$ are integral multiples of $\frac{1}{c_i}$. Recall 
that for each point $y\, \in\, \rho^{-1}(x_i)$, the order of ramification of $\rho$ at $y$ is 
$c_i$. Consequently, there is a unique $\Gamma$--equivariant vector bundle ${\mathcal E}_r$ on 
$Y$ of rank $r$ that corresponds to the parabolic vector bundle $\text{Sym}^{r-1}(E_*) \,=\, 
{\mathbb E}_r$ \cite{Bi}.

This action of $\Gamma$ on ${\mathcal E}_r$ induces an action of $\Gamma$ on the holomorphic line 
bundle $\bigwedge^r {\mathcal E}_r\,=\, \det {\mathcal E}_r$. Since the $\Gamma$--equivariant 
bundle ${\mathcal E}_r$ corresponds to the parabolic bundle ${\mathbb E}_r$, the 
$\Gamma$--equivariant line bundle $\det {\mathcal E}_r$ corresponds to the parabolic line bundle 
$\det {\mathbb E}_r$. As noted in Section \ref{se5.2}, the parabolic line bundle $\det {\mathbb 
E}_r$ is the trivial line bundle with the trivial parabolic structure. Therefore, the 
$\Gamma$--equivariant line bundle $\det {\mathcal E}_r$ corresponding to ${\mathbb E}_r$ is the 
trivial line bundle on $Y$ equipped with the trivial action of $\Gamma$; this means that $\Gamma$ 
act diagonally on the trivial line bundle $Y\times {\mathbb C}$ using the trivial action of 
$\Gamma$ on $\mathbb C$ and the Galois action of $\Gamma$ on $Y$.

\begin{lemma}\label{lemn1}
The holomorphic vector bundle ${\mathcal E}_r$ admits a $\Gamma$--invariant holomorphic 
connection $D_\Gamma$ such that the holomorphic connection on $\det {\mathcal E}_r\,=\, 
{\mathcal O}_Y$ induced by $D_\Gamma$ coincides with the de Rham differential $d$.
\end{lemma}

\begin{proof}
Since the parabolic bundle ${\mathbb E}_r$ corresponds to the $\Gamma$--equivariant bundle
${\mathcal E}_r$, there is a natural bijection between the connections on ${\mathbb E}_r$
and the $\Gamma$--invariant holomorphic connections on ${\mathcal E}_r$ (see \cite{Bi2}).
Now, Corollary \ref{cor3} says that ${\mathbb E}_r$ admits a connection. Hence ${\mathcal E}_r$
admits a $\Gamma$--invariant holomorphic connection. Let $D'$ be a 
$\Gamma$--invariant holomorphic connection on ${\mathcal E}_r$. Note that any holomorphic
connection on the trivial bundle ${\mathcal O}_Y$ is of the form $d+\omega$, where $\omega$
is a holomorphic $1$-form on $Y$. Let
$d+\omega$ be the holomorphic connection on $\det {\mathcal E}_r\,=\, {\mathcal O}_Y$ induced by 
$D'$. Now it is straight-forward to check that
$$
D_\Gamma\, :=\, D' -\frac{1}{r}\omega
$$
is a $\Gamma$--invariant holomorphic connection on ${\mathcal E}_r$ satisfying the condition that
the holomorphic connection on $\det {\mathcal E}_r\,=\, 
{\mathcal O}_Y$ induced by $D_\Gamma$ coincides with the de Rham differential $d$.
\end{proof}

\begin{proposition}\label{nprop}
Let $D$ be a holomorphic connection on ${\mathcal E}_r$ satisfying 
the condition that the holomorphic connection on $\det {\mathcal E}_r$ 
induced by $D$ coincides with the de Rham differential $d$. Then $D$ is
a ${\rm SL}_r$--oper on $Y$.
\end{proposition}

\begin{proof}
The filtration of $E_*$ in \eqref{e4} produces a filtration of subbundles of
the parabolic symmetric product
\begin{equation}\label{efj}
0\,=\, F_0\, \subset\, F_1\, \subset\, F_2\, \subset\, \cdots\,\subset\,
F_{r-1}\, \subset\, F_r\,=\, {\mathbb E}_{r,0}\, ,
\end{equation}
where ${\mathbb E}_{r,0}$ is the holomorphic vector bundle underlying the parabolic
bundle ${\mathbb E}_r$. Each $F_j$ is a holomorphic subbundle of ${\mathbb E}_{r,0}$
of rank $j$. We shall describe the quotient line bundles $F_j/F_{j-1}$ equipped with the
parabolic structure induced by the parabolic structure of ${\mathbb E}_r$.

Let $\mathbb K$ be the parabolic line bundle whose underlying holomorphic line bundle is
the canonical bundle $K_X$, and the parabolic weight of $\mathbb K$ at any $x_i\, \in\,
S$ is $\frac{c_i-1}{c_i}$. If $K'$ is the line subbundle $K^{1/2}_X\otimes {\mathbb L}$
in \eqref{e4} equipped with the parabolic structure induced by $E_*$ (this simply means
that the parabolic weight of $K'$ at any $x_i\, \in\, S$ is $\frac{2c_i-1}{2c_i}$), then
$\mathbb K$ is identified with the parabolic tensor product $K'\otimes K'$.

For any $1\, \leq\, j\, \leq\, r$, the quotient $F_j/F_{j-1}$ in \eqref{efj} equipped with the
parabolic structure induced by the parabolic structure of ${\mathbb E}_r$ is identified with
the parabolic tensor product (and dual in case the exponent is negative)
${\mathbb K}^{(r+1)/2 -j}$ (recall that $r$ is an odd integer).

Since the parabolic weight of $\mathbb K$ at any $x_i\, \in\, S$ is an integral multiple of 
$\frac{1}{c_i}$, there is a unique $\Gamma$--equivariant line bundle on $Y$ that corresponds 
to the parabolic line  bundle $\mathbb K$. It is straight-forward to check that this
$\Gamma$--equivariant line bundle on $Y$ is the canonical bundle $K_Y$ equipped with the
action of $\Gamma$ induced by the action of $\Gamma$ on $Y$.

It was noted above that the quotient $F_j/F_{j-1}$ in \eqref{efj} equipped with the induced 
parabolic structure is ${\mathbb K}^{(r+1)/2 -j}$. Since the correspondence between parabolic 
bundles and equivariant bundles is compatible with the tensor product and dualization 
operations, the parabolic line bundle $F_j/F_{j-1}$ corresponds to the $\Gamma$--equivariant 
line bundle $K^{(r+1)/2 -j}_Y$, because $\mathbb K$ corresponds to $K_Y$.

Consequently, from \eqref{efj} we conclude that that $\Gamma$--equivariant bundle ${\mathcal E}_r$
has a filtration of holomorphic subbundles
\begin{equation}\label{efj2}
0\,=\, V_0\, \subset\, V_1\, \subset\, V_2\, \subset\, \cdots\,\subset\,
V_{r-1}\, \subset\, V_r\,=\, {\mathcal E}_r
\end{equation}
such that for all $1\, \leq\, j\, \leq\, r$,
\begin{itemize}
\item $\text{rank}(V_j)\,=\, j$,

\item $V_j/V_{j-1}\,=\, K^{(r+1)/2 -j}_Y$, and

\item the action of $\Gamma$ on ${\mathcal E}_r$ preserves the subbundle $V_j$.
\end{itemize}

Let $D$ be a holomorphic connection on ${\mathcal E}_r$. From \eqref{efj2} it follows that
$D(V_1)\, \subset\, V_2\otimes K_Y$, and more generally, using induction,
$$
D(V_j)\, \subset\, V_{j+1}\otimes K_Y
$$
for all all $1\, \leq\, j\, \leq\, r-1$. Therefore, $D$ produces an ${\mathcal O}_Y$--linear
homomorphism
\begin{equation}\label{phij}
\phi_j\, :\, V_j/V_{j-1}\, \longrightarrow\, (V_{j+1}/V_j)\otimes K_Y
\end{equation}
for all $1\, \leq\, j\, \leq\, r-1$. Note that both $V_j/V_{j-1}$ and $(V_{j+1}/V_j)\otimes K_Y$
are identified with $K^{(r+1)/2 -j}_Y$, because $V_i/V_{i-1}\,=\, K^{(r+1)/2 -i}_Y$ for all
$1\, \leq\, i\, \leq\, r$. This implies that $\phi_j$ in \eqref{phij} is either an 
isomorphism, or   identically zero.

Assume that $\phi_j\, =\, 0$ for some $1\, \leq\, j\, \leq\, r-1$. Then the connection $D$
on ${\mathcal E}_r$ preserves the subbundle $V_j$. This implies the
\begin{equation}\label{dv}
\text{degree}(V_j)\, =\, 0\, .
\end{equation}
Since $V_i/V_{i-1}\,=\, K^{(r+1)/2 -i}_Y$ for all $1\, \leq\, i\, \leq\, r$, we have
\begin{equation}\label{dv2}
\text{degree}(V_j)\, =\, \, j\frac{r-j}{2}\text{degree}(K_Y)\,=\, j(r-j)(g_Y-1)\, ,
\end{equation}
where $g_Y$ is the genus of $Y$. Since $\text{genus}(X)\, \geq\, 1$, and $S\, \not=\,\emptyset$,
it follows that $g_Y\, >\, 1$. Hence \eqref{dv2} contradicts \eqref{dv}.

In view of the above contradiction we conclude that 
$\phi_j\, \not=\, 0$ for all $1\, \leq\, j\, \leq\, r-1$. As noted before, this implies
that $\phi_j$ is an isomorphism for all $1\, \leq\, j\, \leq\, r-1$.

Therefore, the holomorphic connection $D$ on ${\mathcal E}_r$ defines a $\text{SL}_r$--oper if 
the holomorphic connection on $\det {\mathcal E}_r$ induced by $D$ coincides with the de Rham 
differential $d$ on ${\mathcal O}_Y$. This completes the proof of the proposition.
\end{proof}

The action of $\Gamma$ on
${\mathcal E}_r$ produces an action of $\Gamma$ on the
space of all holomorphic connections on ${\mathcal E}_r$. This action
of $\Gamma$ on the space of all holomorphic connections on ${\mathcal E}_r$
evidently preserves the space ${\mathcal D}({\mathcal E})$ consisting
of all holomorphic connections on ${\mathcal E}_r$
that induce the trivial connection on $\det {\mathcal E}_r\,=\,
{\mathcal O}_Y$.

Since the ${\rm SL}_r$--opers on $Y$ are the equivalence classes of holomorphic connections on 
$\mathcal E$ lying in ${\mathcal D}({\mathcal E})$, we obtain an action of $\Gamma$ on the space 
of all ${\rm SL}_r$--opers on $Y$. A ${\rm SL}_r$--oper on $Y$ is called 
$\Gamma$--\textit{invariant} if it is fixed by this action of $\Gamma$ on $\mathcal E$.

\begin{theorem}\label{thm2}
Parabolic ${\rm SL}_r$--opers on $X$ are in a natural bijection with the
$\Gamma$--invariant ${\rm SL}_r$--opers on $Y$.
\end{theorem}

\begin{proof}
As noted before,
since the $\Gamma$--equivariant bundle ${\mathcal E}_r$ corresponds to
the parabolic bundle $\text{Sym}^{r-1}(E_*)\,=\, {\mathbb E}_r$, the parabolic
connections on ${\mathbb E}_r$ are in a natural bijective correspondence with the
$\Gamma$--invariant holomorphic connections on ${\mathcal E}_r$.
We recall that the ${\rm SL}_r$--opers on $Y$ are the
equivalence classes of holomorphic connections
$D'$ on ${\mathcal E}_r$ such that the connection on 
$\bigwedge^r{\mathcal E}_r\, =\, {\mathcal O}_Y$ induced by $D'$ coincides with the
one given by the de Rham differential $d$. For such a $\Gamma$--invariant holomorphic connection
$D'$ on ${\mathcal E}_r$, the
connection $D$ on the parabolic bundle ${\mathbb E}_r$ has the property that the connection
of the parabolic line bundle $\bigwedge^r {\mathbb E}_r$ induced by $D$ is the trivial
connection on ${\mathcal O}_X$ (the connection on ${\mathcal O}_X$ given
by the de Rham differential $d$). The theorem follows from these.
\end{proof}

A projective structure $P_Y$ on $Y$ is called $\Gamma$--invariant if the
automorphisms of $Y$ given by the action of $\Gamma$ on $Y$ preserve
$P_Y$. Let ${\mathcal P}(Y)^\Gamma$ denote the space of all projective structures on $Y$
that are $\Gamma$--invariant. Any element of ${\mathcal P}(Y)^\Gamma$ defines a logarithmic
projective structure on $X$ (see \cite{BDM}).
This ${\mathcal P}(Y)^\Gamma$ is an affine space for
$$H^0(Y,\, K^{\otimes 2}_Y)^\Gamma\,=\, H^0(X,\, K^{\otimes 2}_X\otimes {\mathcal O}_X(S))\, .$$

The space of all ${\rm SL}_r$--opers on $Y$ is in bijection with
$$
{\mathcal P}(Y)\times (\bigoplus_{i=3}^r H^0(Y,\, K^{\otimes i}_Y))
$$
(see \cite[Theorem 4.9]{Bi3} and \cite[Eq.~(5.4)]{Bi3}). Consequently, the space of all
$\Gamma$--invariant ${\rm SL}_r$--opers on $Y$ is in
bijection with
$${\mathcal P}(Y)^\Gamma\times (\bigoplus_{i=3}^r H^0(Y,\, K^{\otimes i}_Y)^\Gamma)
\,=\,{\mathcal P}(Y)^\Gamma\times (\bigoplus_{i=3}^r 
H^0(X,\, K^{\otimes i}_X\otimes {\mathcal O}_X((i-1)S)))\, .$$
Now from Theorem \ref{thm2} we conclude that the
space of all parabolic ${\rm SL}_r$--opers on $X$ is in bijection with
$$
{\mathcal P}(Y)^\Gamma\times (\bigoplus_{i=3}^r 
H^0(X,\, K^{\otimes i}_X\otimes {\mathcal O}_X((i-1)S)))\, .$$

Setting $r\,=\,2$ we get that the parabolic ${\rm SL}_2$--opers on $X$ are
identified with ${\mathcal P}(Y)^\Gamma$. Consequently, we have the following:

\begin{theorem}\label{propl}
The space of all parabolic ${\rm SL}_r$--opers on $X$ is in a natural bijection with
$$
{\mathcal P}_2\times (\bigoplus_{i=3}^r 
H^0(X,\, K^{\otimes i}_X\otimes {\mathcal O}_X((i-1)S)))\, ,$$
where ${\mathcal P}_2$ denotes the space of all parabolic ${\rm SL}_2$--opers on $X$.
\end{theorem}

\section{Further comments}

There are other natural generalizations of ${\rm SL}_r$--opers 
to the parabolic set-up (see also \cite{ABF}). Following the original definition
of \cite{BD} one can also define a parabolic ${\rm SL}_r$--oper
for a given parabolic divisor $S = \sum_i x_i$, with 
$x_i \in X$, and with real
parabolic weights $\alpha_{i,j}$ such that
$$ 0 \leq \alpha_{i,1} < \cdots < \alpha_{i,j} < \cdots <
\alpha_{i,r} < 1 \ \ \text{for each} \ \ x_i \in S, $$
and $m_i = \sum_j \alpha_{i,j} \in \mathbb{N}^*$, as a triple
$(E,E_\bullet, D)$, where
\begin{itemize}
\item $E$ is a rank $r$ vector bundle,

\item $E_\bullet \ : \ 0 = E_0 \subset E_1 \subset \cdots
\subset E_r = E$ is a filtration by subbundles such that $\mathrm{rk}(E_i) = i$,

\item $D : E \rightarrow E \otimes K_X \otimes {\mathcal O}_X(S)$ is a logarithmic connection on $E$ such that
$D(E_i) \subset E_{i+1} \otimes K_X \otimes 
{\mathcal O}_X(S)$ and the induced ${\mathcal O}_X$-linear
map $\overline{D} : E_i/E_{i-1} \rightarrow E_{i+1}/E_{i}
\otimes K_X \otimes {\mathcal O}_X(S)$ factorizes through
an isomorphism $E_i/E_{i-1} \cong E_{i+1}/E_{i}
\otimes K_X$ followed by the natural sheaf inclusion
$E_{i+1}/E_{i} \otimes K_X \subset E_{i+1}/E_{i}
\otimes K_X \otimes {\mathcal O}_X(S)$.

\item For each $x_i \in S$ the residue map 
$\mathrm{Res}(D,x_i)$ preserves the full flag $(E_\bullet)_{x_i}$ and acts as multiplication by 
$\alpha_{i,r+1-j}$ on $(E_j/E_{j-1})_{x_i}$

\item $\mathrm{det} E = {\mathcal O}_X( - \sum_i m_i x_i)$
and $\mathrm{det} D$ is the logarithmic de Rham differential
$f \mapsto df$ restricted to ${\mathcal O}_X( - \sum_i m_i x_i)$.
\end{itemize}

It can be easily shown that for $r=2$ and rational parabolic weights the two definitions coincide 
(up to a choice of a theta-characteristic and a square-root of ${\mathcal O}_X( - S)$). For $r \,>\, 
2$ we note that the above definition is more general than Definition \ref{def2}. It would be 
interesting to know whether they coincide for the special rational parabolic weights of the ${\rm 
SL}_r$--opers in Section \ref{se6} and whether Theorem \ref{propl} holds for general real weights. We 
will address these questions in a future paper.

\section*{Acknowledgements}

This work has been supported by the French government through the UCAJEDI Investments in the Future 
project managed by the National Research Agency (ANR) with the reference number ANR2152IDEX201. The 
first author is partially supported by a J. C. Bose Fellowship, and school of mathematics, TIFR, is 
supported by 12-R$\&$D-TFR-5.01-0500\,.



\begin{thebibliography}{ZZZZ}

\bibitem[ABF]{ABF} M. Alim, F. Beck and L. Fredrickson, Parabolic Higgs bundles, $tt^*$ 
connections and opers, arXiv:1911.06652v1.

\bibitem[At1]{At1} M. F. Atiyah, On the Krull-Schmidt theorem with application
to sheaves, \textit{Bull. Soc. Math. Fr.} \textbf{84} (1956), 307--317.

\bibitem[At2]{At2} M. F. Atiyah, Complex analytic connections in fibre
bundles, \textit{Trans. Amer. Math. Soc.} \textbf{85} (1957), 181--207.

\bibitem[BD]{BD} A. Beilinson and V. G. Drinfeld, Opers,
{\em arXiv math/0501398}.

\bibitem[Bi1]{Bi} I. Biswas, Parabolic bundles as orbifold bundles, {\it Duke Math. 
Jour.} {\bf 88} (1997), 305--325.

\bibitem[Bi2]{Bi2} I. Biswas, A criterion for the existence of a flat connection on a parabolic
vector bundle, {\it Adv. in Geom.} {\bf 2} (2002), 231--241.

\bibitem[Bi3]{Bi3} I. Biswas, Invariants for a class of equivariant immersions of the universal
cover of a compact {R}iemann surface into a projective space,
{\it Jour. Math. Pures Appl.} {\bf 79} (2000), 1--20.

\bibitem[BL]{BL} I. Biswas and M. Logares, Connection on parabolic vector bundles over
curves, {\it Inter. Jour. Math.} {\bf 22} (2011), 593--602.

\bibitem[BDM]{BDM} I. Biswas, S. Dumitrescu and B. McKay, Logarithmic Cartan geometry on
complex manifolds, {\it Jour. Geom. Phy.} {\bf 148} (2020) 103542.

\bibitem[Bo1]{Bo1}
N. Borne, Fibr\'es paraboliques et champ des racines,
{\em Int. Math. Res. Not. IMRN}, {\bf 16}, Art. ID rnm049, 38, (2007).

\bibitem[Bo2]{Bo2}
N. Borne, Sur les repr\'esentations du groupe fondamental d'une vari\'et\'e
priv\'ee d'un diviseur \`a croisements normaux simples,
{\em Indiana Univ. Math. Jour.} {\bf 58} (2009), 137--180.


\bibitem[De]{De} P. Deligne, {\it \'Equations diff\'erentielles \`a points
singuliers r\'eguliers}, Lecture Notes in Mathematics, Vol. 163, Springer-Verlag,
Berlin-New York, 1970.

\bibitem[DS1]{DS} V. G. Drinfeld and V. V. Sokolov, Lie algebras and equations of Korteweg-de
Vries type. (Russian) Current problems in mathematics, Vol. 24, 81--180,
Itogi Nauki i Tekhniki, Akad. Nauk SSSR, Vsesoyuz. Inst. Nauchn. i Tekhn. Inform., Moscow, 1984.

\bibitem[DS2]{DS2} V. G. Drinfeld and V. V. Sokolov, Equations of Korteweg-de Vries type, and
simple Lie algebras. (Russian) {\it Dokl. Akad. Nauk SSSR} {\bf 258} (1981), 11--16.


\bibitem[GKM]{GKM} D. Gallo, M. Kapovich and A. Marden, The monodromy groups of Schwarzian 
equations on closed Riemann surfaces, \textit{Annals of Math.} {\bf 151}, (2000), 625--704.

\bibitem[Gr]{Gr} A. Grothendieck, Sur la classification des fibr\'es holomorphes
sur la sph\`ere de Riemann, {\it Amer. Jour. Math.} {\bf 79} (1957), 121--138.

\bibitem[Gu]{Gu} R. C. Gunning, {\it On uniformization of complex manifolds: the 
role of connections}, Princeton Univ. Press, 1978.

\bibitem[MY]{MY} M. Maruyama and K. Yokogawa, Moduli of
parabolic stable sheaves, \textit{Math. Ann.} \textbf{293}
(1992), 77--99.

\bibitem[MS]{MS} V. B. Mehta and C. S. Seshadri, Moduli of vector bundles on curves
with parabolic structure, {\it Math. Ann.} \textbf{248} (1980), 205--239.

\bibitem[Na]{Na} M.~Namba, {\it Branched coverings and algebraic
functions}, Pitman Research Notes in Mathematics Series, 161,
Longman Scientific $\&$ Technical, Harlow; John Wiley $\&$ Sons,
Inc., New York, 1987.

\bibitem[Oh]{Oh} M. Ohtsuki, A residue formula for Chern classes associated with
logarithmic connections, \textit{Tokyo Jour. Math.} \textbf{5} (1982), 13--21.

\bibitem[Se]{Se} C. S. Seshadri, \textit{Fibr\'es vectoriels sur les
courbes alg\'ebriques}, Ast\'erisque, No. 96, Soci\'et\'e Math.
de Fr., 1982.

\bibitem[W]{W} Y. Wakabayashi, A theory of dormant opers on pointed stable curves --- a proof of Joshi's conjecture, {\em arXiv:math.AG/1411.1208}.

\end{thebibliography}
\end{document}